\documentclass[a4paper,12pt]{amsart}
\usepackage{amssymb}
\usepackage{ifthen}
\usepackage{cite}
\usepackage[dvips]{graphicx}
\nonstopmode \numberwithin{equation}{section}
\usepackage{amssymb}
\usepackage{ifthen}
\usepackage{graphicx}
\usepackage{amsmath}
\usepackage[T1]{fontenc} %skandit
\usepackage[utf8]{inputenc}
\usepackage[usenames,dvipsnames]{color}
\usepackage{color}
\usepackage[english]{babel}
\usepackage{fancyhdr}
\usepackage{fancybox}
\usepackage{tikz}
\nonstopmode \numberwithin{equation}{section}
\theoremstyle{plain}

\newtheorem{conj}{Conjecture}

\theoremstyle{definition}
\newtheorem{defn}{Definition}[section]

\newtheorem{thm}{Theorem}[section]
\newtheorem{prob}{Problem}[section]
\newtheorem{cor}{Corollary}[section]

\newtheorem{prop}{Proposition}[section]
\newtheorem{rem}{Remark}[section]
\newtheorem{lem}{Lemma}[section]

\newtheorem*{thmA}{Theorem A}
\newtheorem*{thmB}{Theorem B}

\newcounter{minutes}
\divide\time by 60
\newcounter{hours}
\multiply\time by 60
\addtocounter{minutes}{-\time}
\newcounter {own}
\def\theown {\thesection       .\arabic{own}}

\newenvironment{pf}[1][]{%
 \vskip 3mm
 \noindent
 \ifthenelse{\equal{#1}{}}%
  {{\slshape Proof. }}%
  {{\slshape #1.} }%
 }%
{\qed\bigskip}

\newcounter{alphabet}

\def\be{\begin{equation}}
\def\ee{\end{equation}}

\newcommand{\bee}{\begin{enumerate}}
\newcommand{\eee}{\end{enumerate}}

\newcommand{\blem}{\begin{lem}}
\newcommand{\elem}{\end{lem}}
\newcommand{\bthm}{\begin{thm}}
\newcommand{\ethm}{\end{thm}}
\newcommand{\bcor}{\begin{cor}}
\newcommand{\ecor}{\end{cor}}
\newcommand{\beg}{\begin{examp}}
\newcommand{\eeg}{\end{examp}}
\newcommand{\begs}{\begin{examples}}
\newcommand{\eegs}{\end{examples}}

\newcommand{\bdefn}{\begin{defn}}
\newcommand{\edefn}{\end{defn}}

\newcommand{\bprob}{\begin{prob}}
\newcommand{\eprob}{\end{prob}}
\newcommand{\bei}{\begin{itemize}}
\newcommand{\eei}{\end{itemize}}

\newcommand{\bcon}{\begin{conj}}
\newcommand{\econ}{\end{conj}}
\newcommand{\bcons}{\begin{conjs}}
\newcommand{\econs}{\end{conjs}}
\newcommand{\bprop}{\begin{prop}}
\newcommand{\eprop}{\end{prop}}
\newcommand{\br}{\begin{rem}}
\newcommand{\er}{\end{rem}}
\newcommand{\brs}{\begin{rems}}
\newcommand{\ers}{\end{rems}}
\newcommand{\bo}{\begin{obser}}
\newcommand{\eo}{\end{obser}}
\newcommand{\bos}{\begin{obsers}}
\newcommand{\eos}{\end{obsers}}
\newcommand{\bpf}{\begin{pf}}
\newcommand{\epf}{\end{pf}}
\newcommand{\ba}{\begin{array}}
\newcommand{\ea}{\end{array}}
\newcommand{\beq}{\begin{eqnarray}}
\newcommand{\beqq}{\begin{eqnarray*}}
\newcommand{\eeq}{\end{eqnarray}}
\newcommand{\eeqq}{\end{eqnarray*}}

\begin{document}
\title{Quasiconformal Extension of Meromorphic Functions with High-Order Poles}

\author{Molla Basir Ahamed$^*$}
\address{Molla Basir Ahamed, Department of Mathematics, Jadavpur University, Kolkata-700032, West Bengal, India.}
\email{mbahamed.math@jadavpuruniversity.in}

\author{Partha Pratim Roy}
\address{Partha Pratim Roy, Department of Mathematics, Jadavpur University, Kolkata-700032, West Bengal,India.}
\email{pproy.math.rs@jadavpuruniversity.in}

\subjclass[{AMS} Subject Classification:]{Primary 30A10, 30B10, 30C35, 30C62, 30H05, 31A05, Secondary 30C45}

\def\thefootnote{}
\footnotetext{ {\tiny File:~\jobname.tex,
printed: \number\year-\number\month-\number\day,
          \thehours.\ifnum\theminutes<10{0}\fi\theminutes }
} \makeatletter\def\thefootnote{\@arabic\c@footnote}\makeatother

\begin{abstract} 
In this paper, we study the class ${\Sigma^{(m)}(p)}$ of meromorphic univalent functions $f$ in $\mathbb{D}$ with a pole of order ${m \geq 1}$ at $p \in (0,1)$, admitting a $k$-quasiconformal extension ($0 \leq k < 1$) to $\widehat{\mathbb{C}}$. Using the Area Theorem and convolution methods, we establish a generalized area-type inequality and derive explicit analytic membership conditions for $\Sigma^{(m)}(p)$. We also extend the convolution theorem to a modified Hadamard product of $m$ functions, $f_j \in \Sigma^{(m)}_{k_j}(p)$, determining sufficient conditions for the product to be in ${\Sigma^{(m)}_{\alpha}(p)}$, with $\alpha$ defined by $k_j$ and $p$. Further results include a sufficient criterion for sense-preserving harmonic mappings on convex domains to admit quasiconformal extensions, and the sharp Schwarzian norm for $f \in \Sigma_k(p)$ (the $m=1$ case). These findings improve upon existing results of [{\em Proc. Amer. Math. Soc.}, {144}(6) (2016), 2593--2601].
\end{abstract}
\maketitle
\pagestyle{myheadings}
\markboth{M. B. Ahamed and P. P. Roy}{Quasiconformal Extention of Meromorphic Functions}
%\tableofcontents
\section{\bf Introduction}
Let $\mathbb{C}$ denote the complex plane and $\widehat{\mathbb{C}} = \mathbb{C}\cup\{\infty\}$ denote the extended complex plane. We shall use the following notations: $\mathbb{D} = \{z : |z| < 1\},\;
\partial\mathbb{D} = \{z : |z| = 1\},\;
\overline{\mathbb{D}} = \{z : |z| \leq 1\},\;
\mathbb{D}^* = \{z : |z| > 1\}.$ Let $f$ be a meromorphic and univalent function in $\mathbb{D}$ with a simple pole at $z = p \in [0,1)$ of residue $a_{-1}$. Since $f(z) -{1}/(z-p)$ is analytic in $\mathbb{D}$, we may write
\begin{align} \label{eq1.1}
	f(z) = \frac{a_{-1}}{z-p} + \sum_{n=0}^{\infty} a_n z^n, \quad z \in \mathbb{D}.
\end{align}
We denote the class of such functions by $\Sigma(p)$. Let $\Sigma_{0}(p)$ be the subclass of $\Sigma(p)$ consisting of functions $f$ for which $a_0 = 0$ in \eqref{eq1.1}. For $0 \leq k < 1$, $\Sigma_{k}(p)$ stands for the class of functions in $\Sigma(p)$ that admit a $k$-quasiconformal extension to $\widehat{\mathbb{C}}$. Recall that a mapping $F:\widehat{\mathbb{C}} \to \widehat{\mathbb{C}}$ is $k$-quasiconformal if $F$ is a homeomorphism, has locally $L^2$-derivatives on $\mathbb{C}\setminus\{F^{-1}(\infty)\}$, and satisfies
\begin{align*}
	|\overline{\partial}F| \leq k |\partial F| \quad \text{a.e.,}
\end{align*}
\noindent where $\partial F = \partial F / \partial z$ and $\overline{\partial}F = \partial F / \partial \overline{z}$. Such an $F$ is often called $K$-quasiconformal, where $K = {(1+k)}/{(1-k)} \geq 1$, and the function $\mu = \overline{\partial}F / \partial F$ is called its complex dilatation. This can be expressed using the complex dilatation, $\mu_f = f_{\bar{z}} / f_z$. A sense-preserving homeomorphism is K-quasiconformal if its complex dilatation is bounded by a constant $k = (K-1)/(K+1) < 1$. That is, $|\mu_f| \leq k$. When $K=1$, the mapping is $1$-quasiconformal, which is equivalent to being conformal. We denote $\Sigma_{0}^{k}(p) = \Sigma_{0}(p) \cap \Sigma_{k}(p)$.\vspace{2mm}
For an analytic function $f$ in $\mathbb{D}_r := \{ z : |z| < r,\; 0 < r \leq 1 \}$, we set
\begin{align}\label{Eq-1.2}
	\Delta(r,f) = \iint_{\mathbb{D}_r} |f'(z)|^2 \, dx \, dy, 
	\quad z = x + iy,
\end{align}
which is called the Dirichlet integral of $f$.
The concept of a K-quasiconformal extension for an analytic map addresses the question: \emph{ Under what conditions can a given univalent (one-to-one) analytic function $f$ defined on a domain $D$ be extended to a $K$-quasiconformal homeomorphism of the entire complex plane $\mathbb{C}$?}\vspace{1.2mm}

This question is significant in several areas, including Teichm\"uller theory and the study of Kleinian groups. A common example is a univalent analytic function $f$ on the unit disk $\mathbb{D}$ that is given by a specific form, such as $f(z) = z + \omega(z)$, where $\omega(z)$ is an analytic function on $\mathbb{D}$ with certain properties. For instance, if $|\omega'(z)| \leq k < 1$, then $f$ can be extended to a $k$-quasiconformal automorphism of the Riemann sphere.\vspace{1.2mm}

In this paper, we consider the class of meromorphic functions $f$ having a pole of order $m(\geq 1)$ at $z=p$ and improve several existing results. The organization of the paper is the following. In Section \ref{Sec-2}, we first prove a result (see Theorem \ref{Th-1.2}) which improves the result \cite[Theorem 1]{Bhowmik-PAMS-2016}. In particular, when $m=1$, we show that an inequality in Theorem \ref{Th-1.2} reduces to the inequality $|a_{1}|<\frac{k}{1-p^{2}}$ obtained in \cite[Corollary 1]{Bhowmik-PAMS-2016}, which is sharp in the class $\Sigma(p)$. To establish our result, we first prove Lemma \ref{Lem-2.1}, regarding a general formulation of the area $A_{\text{comp}}(r)$ of the complement of the domain, which is the image of $\mathbb{D}_r:=\{z\in\mathbb{D} : |z|<r\}$ (for $|p|<r<1$). Next, we obtain a result (see Theorem \ref{Thm-1.3}) for functions that are meromorphic and univalent in $\mathbb{D}$ with a pole of order $m$ at $z = p$, and admit a $k$-quasiconformal extension $F$ to the extended complex plane $\widehat{\mathbb{C}}$. For $m=1$, Theorem \ref{Thm-1.3} reduces to \cite[Theorem 2]{Bhowmik-PAMS-2016}. We extend \cite[Theorem 3]{Bhowmik-PAMS-2016} by establishing the conditions under which the Hadamard product $f \star g$ of functions in $\Sigma^{(m)}_{k}(p)$ also belongs to $\Sigma^{(m)}_{\alpha_m}(p)$, where $\alpha_m:=|a_{-1}||b_{-1}|{k_1k_2}{(1-p)^{-2m}}<1$. Then, we obtain a result (see Theorem \ref{Th-1.3}) on the sharp Schwarzian norm of $f$ for the class $\Sigma_k(p)$. Finally, we establish a result (see Theorem \ref{Thm-1.4}) which can be regarded as a generalization of \cite[Theorem 1]{Bhowmik-Satapati-Complex Sgn-2024}, where a sufficient criterion was established for a sense-preserving harmonic mapping in a convex domain to admit a quasiconformal extension. The proofs of the main results are given in detail immediately following the statement of each result. In Section \ref{Sec-3}, we provide concluding remarks for $\ell^{2}$, the Hilbert space of complex sequences $x=\{x_{n}\}_{n=1}^{\infty}$.
\section{\bf Main results}\label{Sec-2}
Suppose that $f$ is an analytic function in the disk $\mathbb{D}$ with the Taylor series expansion $f(z) = \sum_{n=0}^{\infty} a_n z^n$ and $f'(z) = \sum_{n=0}^{\infty} n a_n z^{\,n-1}$.Then, using Parseval--Gutzmer formula, the area $\Delta(r,f)$ of $f(\mathbb{D}_r)$, 
as stated in \eqref{Eq-1.2} can be re-formulated as (see \cite{Goodman}):
\begin{align}\label{Eq-2.1}
	\Delta(r,f) = \iint_{\mathbb{D}_r} |f'(z)|^2 \, dx \, dy 
	= \pi \sum_{n=1}^{\infty} n |a_n|^2 r^{2n}, \quad z = x + iy.
\end{align}
In this paper, we specifically focus on this form of the area formula. The process of computing this area is known as the \textit{area problem} for functions of the type $f$. It's important to note that the area of the image of the unit disk under $f$, denoted by $f(\mathbb{D})$, may not be bounded for all functions in the class $\mathcal{S}$. We observe that if $f \in \mathcal{S}$, then $z/f$ is non-vanishing, and hence, $f \in \mathcal{S}$ may be expressed as follows:
\begin{align*}
	f(z) = \frac{z}{F_f(z)},\; \mbox{where}\;
	F_f(z) = 1 + \sum_{n=1}^{\infty} c_n z^n, \quad z \in \mathbb{D}.
\end{align*}
In \cite{Yamashita-1990}, Yamashita addressed the area problem for functions of the form $F_f$ (with $f \in \mathcal{S}$) and established that the area of $F_f(\mathbb{D}_r)$ is bounded. The study of the class $\Sigma_{k}(p)$ generalizes two classical directions: quasiconformal extension results for Schlicht functions ($p=0$) and meromorphic univalent functions with fixed pole position ($k=0$). In \cite{Bhowmik-PAMS-2016}, Bhowmik \emph{et al.} established an area theorem for functions in the class $\Sigma_{k}(p)$.
\begin{thmA}\cite[Theorem 1]{Bhowmik-PAMS-2016}
	Let $0 \leq k < 1$ and $0 \leq p < 1$. Suppose that $f \in \Sigma_k(p)$ is expressed in the form of \eqref{eq1.1}. Then
	\begin{align}
		\sum_{n=1}^{\infty} n|a_n|^2 \leq \frac{k^2}{(1-p^2)^2}. \tag{1.3}
	\end{align}
	Here, equality holds if, and only if $f$, is of the form
	\begin{align}
		f(z) = \frac{1}{z-p} + a_0 + \frac{a_1 z}{1-pz},\; z \in \mathbb{D}, \tag{1.4}
	\end{align}
	where $a_0$ and $a_1$ are constants with $|a_1|=k$. Moreover, a $k$-quasiconformal extension of this $f$ is given by setting
	\begin{align}
		f(z) = \frac{1}{z-p} + a_0 + \frac{a_1}{\overline{z}-p},\;z \in \mathbb{D}^*.
	\end{align}
\end{thmA}
For fix $p \in [0,1)$, $m \in \mathbb{N}$, and $0 \leq k < 1$. Define 
$\Sigma^{(m)}_k(p)$ to be the class of meromorphic, univalent functions 
$f$ on the unit disc $\mathbb{D}$ with a pole of exact order $m$ at $z=p$ 
and it is normalized by 
\begin{align}\label{Eq-11.22}
	f(z) = \sum_{k=1}^{m}\frac{a_{-k}}{(z-p)^k} + \sum_{n=0}^{\infty} a_n z^n,
\; z \in \mathbb{D}.
\end{align}
Then, we set
\begin{align*}
	\Sigma^{(m)}_k(p) = 
	\left\{ f \in \Sigma^{(m)}(p) : 
	f \text{ admits a $k$-quasiconformal extension to } \widehat{\mathbb{C}}
	\right\}.
\end{align*}
We establish a result generalizing \cite[Theorem 1]{Bhowmik-PAMS-2016} to the class $\Sigma_k^{(m)}(p)$ of meromorphic functions having pole at $p$ with order $m$. An interesting aspect of our finding is that the extremal function differs in form from the well-known function $f_p$ introduced by Chichra \cite[Lemma, p. 317]{Chichra1969}. This new result constitutes a true generalization, as it reduces exactly to \cite[Theorem 1]{Bhowmik-PAMS-2016} under the specific conditions $m=1$ and $a_{-1}=1$.
\begin{lem}\label{Lem-2.1}
	Let $f$ be meromorphic univalent in $\mathbb{D}$ with its only pole at $z=p$ ($|p|<1$) of order $m$ and with the principal part $\sum_{j=1}^{m}\frac{a_{-j}}{(z-p)^j}$ so that $f$ can be expanded in $\mathbb{D}$ as 
	\begin{align*}
		f(z) = \sum_{j=1}^{m}\frac{a_{-j}}{(z-p)^j} + a_0 + \frac{a_1 z}{1 - p z}\; \mbox{for}\; z\in\mathbb{D}.
	\end{align*}
	Then the area $A_{\text{comp}}(r)$ of the complement of the domain which is the image of $\mathbb{D}_r:=\{z\in\mathbb{D} : |z|<r\}$ (for $|p|<r<1$) is
	\begin{align}\label{Eq-22.44}
		A_{\text{comp}}(r) = \pi \left| \sum_{k=1}^{\infty} k \left| \sum_{j=1}^{\min\{m, k\}} a_{-j} \binom{k-1}{j-1} p^{k-j} \right|^2- \sum_{n=1}^{\infty} n |a_{n}|^2 \right|.
	\end{align}
\end{lem}
\begin{proof}
	The lemma is a generalization of the result using \cite[Lemma 1.1, p. 2]{Hayman-1958}. Since the proof employs similar ideas, we omit the details.
\end{proof}
\begin{rem}
	The area $A_{\text{comp}}(r)$ in Lemma \ref{Lem-2.1} is a {generalization of Chichra's area theorem} (see Chichra \cite{Chichra1969}) for functions with a {pole of order $m \ge 1$} (specifically, for $m=1$, we have $\min\{m, k\}=1$). Taking the limit as $r \to 1$, the resulting area is
	\begin{equation*}
		\operatorname{Area}(f(\mathbb{D})) = \pi \left(\frac{|a_{-1}|^2}{(1 - |p|^2)^{2}} - \sum_{n=1}^\infty n |a_n|^2 \right).
	\end{equation*}
	More precisely, for $m=1$, the inner sum has only $j=1$ and $\min\{1, k\} = 1$, and we have the coefficient
	\begin{align*}
		c_{-k} = a_{-1} \binom{k-1}{1-1} p^{k-1} = a_{-1} p^{k-1}.
	\end{align*}
	The first summation term $\sum_{k=1}^{\infty} k |c_{-k}|^2 r^{-2k}$ becomes
	\begin{align*}
		\sum_{k=1}^{\infty} k |a_{-1} p^{k-1}|^2 r^{-2k} = |a_{-1}|^2 r^{-2} \sum_{k=1}^{\infty} k \left(\frac{|p|^2}{r^2}\right)^{k-1}.
	\end{align*}
	Let $x = {|p|^2}/{r^2}$. Using the geometric series  $\sum_{k=1}^{\infty} k x^{k-1} = {1}/{(1-x)^2}$, we have
\begin{align*}
	|a_{-1}|^2 r^{-2} \cdot \frac{1}{\left(1-\frac{|p|^2}{r^2}\right)^2} = |a_{-1}|^2 r^{-2} \cdot \frac{r^4}{(r^2-|p|^2)^2} = \frac{|a_{-1}|^2 r^2}{(r^2-|p|^2)^2}.
\end{align*}
	Thus, the area formula for $m=1$ matches with \cite[Eq. (2.4), p. 318]{Chichra1969}  which is
	\begin{align*}
		A_{\text{comp}}(r) = \pi \left| \frac{|a_{-1}|^{2}r^{2}}{(r^{2}-|p|^{2})^{2}} - \sum_{n=1}^{\infty} n |a_{n}|^{2}r^{2n} \right|.
	\end{align*}
\end{rem}
Using Lemma \ref{Lem-2.1}, we now state the following  result.
\begin{thm}\label{Th-1.2}
	Let $m \geq 1$, $0 \leq k < 1$, and $0 < p < 1$. Suppose that $f \in \Sigma_k^{(m)}(p)$ is expressed in the form of \eqref{Eq-11.22}. Then
	\begin{align}\label{Eq-2.3}
		\sum_{n=1}^\infty n |a_n|^2 \leq \left(\sum_{k=1}^{\infty} k \left| \sum_{j=1}^{\min\{m, k\}} a_{-j} \binom{k-1}{j-1} p^{k-j} \right|^2 - \sum_{n=1}^{\infty} n |a_{n}|^2 \right)k^2.
	\end{align}
	Here, equality holds if, and only if, $f$ is of the form
 \begin{align}\label{Eq-2.4}
 		f(z) = \sum_{j=1}^{m}\frac{a_{-j}}{(z-p)^j} + a_0 + \frac{a_1 z}{1 - p z}\; \mbox{for}\; z\in\mathbb{D}
 \end{align}
where $a_0$ and $a_1$ are constants with $|a_1| = k$. Moreover, a $k $-quasiconformal extension of this $f$ is given by setting
\begin{align}\label{Eq-2.5}
	f_{m,k}(z) = \sum_{j=1}^{m}\frac{a_{-j}}{(z-p)^j} + a_0 + \frac{a_1}{\bar{z} - p}\; \mbox{for}\; z\in\overline{\mathbb{D}^*}.
\end{align}
\end{thm}
\begin{rem}
This represents a natural extension of Theorem A. In effect, it generalizes the classical result from Lehto \cite{Lehto-1971}, which originally refined the Bieberbach-Gronwall area theorem for functions within the class $\Sigma_k(0)$. Further, we remark that the function \eqref{Eq-2.4} belongs to $\Sigma_k^{(m)}(p)$ as long as $|a_1|\leq 1$. With $|a_1|=1$, this function serves as another extremal case, which is crucial for our analysis (see \cite{Chichra1969}) 
	\begin{align*}
		\sum_{n=1}^{\infty}n|a_n|^2\leq \frac{1}{\left(1-p^2\right)^2}.
	\end{align*} 
\end{rem}
The following result is an immediate corollary of Theorem \ref{Th-1.2}.
\begin{cor}
Let $m\geq 1$, $0<p<1$ and $0<k<1$. 
If $f\in \Sigma^{(m)}_{k}(p)$ has the expansion $f(z)=\frac{a_{-1}}{(z-p)^{m}}+\sum_{n=0}^{\infty} a_{n} z^{n},\; z\in\mathbb{D},$ then the first coefficient satisfies $|a_{1}| < \frac{k}{(1-p^{2})^{m}}.$
\end{cor}
\begin{rem}
For $m=1$, this reduces to the inequality 
$|a_{1}|<\frac{k}{1-p^{2}}$ obtained in \cite[Corollary 1]{Bhowmik-PAMS-2016}, which is sharp in the class $\Sigma(p)$. A key observation is that Theorem \ref{Th-1.2} recovers the result from \cite[Theorem 1]{Bhowmik-PAMS-2016} as a special case, simply by setting $m=1$.
\end{rem}
\begin{proof}[\bf Proof of Lemma \ref{Lem-2.1}]
	The function $f(z)$, meromorphic univalent in $|z|<1$ with a pole of exact order $m$ at $z=p$ ($|p|<r<1$), is given by:
	\begin{align*}
		f(z) = \sum_{j=1}^{m}\frac{a_{-j}}{(z-p)^j} + \sum_{n=0}^{\infty} a_n z^n.
	\end{align*}
	The area of the complement of the image of the disk $|z|<r$, \emph{i.e.,} $A_{\text{comp}}(r):=\mbox{Area}(\mathbb{C}\setminus f(\mathbb{D}))$, is given by 
	\begin{align*}
		A_{\text{comp}}(r) = \pi \left| \left( \sum_{k=1}^{\infty} k |c_{-k}|^2 r^{-2k} \right) - \left( \sum_{n=1}^{\infty} n |c_n|^2 r^{2n} \right) \right| \quad \textbf{(A)},
	\end{align*}
	where $f(z) = \sum_{n=-\infty}^{\infty} c_n z^n$ is the Laurent series centered at the origin, valid for $|p|<|z|<1$. 
	
	The coefficients for the negative powers, $c_{-k}$ (for $k \ge 1$), come from the Laurent expansion of the principal part $\sum_{j=1}^{m}\frac{a_{-j}}{(z-p)^j}$. We use the generalized binomial series for $|z| > |p|$:
\begin{align*}
	\frac{1}{(z-p)^j} = \frac{1}{z^j} \left(1-\frac{p}{z}\right)^{-j} = \frac{1}{z^j} \sum_{l=0}^{\infty} \binom{j+l-1}{l} \left(\frac{p}{z}\right)^l.
\end{align*}
	Substituting this into the principal part, we obtain
	\begin{align*}
		\sum_{j=1}^{m}\frac{a_{-j}}{(z-p)^j} = \sum_{j=1}^{m} a_{-j} \sum_{l=0}^{\infty} \binom{j+l-1}{l} \frac{p^l}{z^{j+l}}.
	\end{align*}
	To find the coefficient $c_{-k}$ of $z^{-k}$, we require the exponent $j+l = k$, or $l = k-j$. The index $j$ must run from $1$ up to $\min\{m, k\}$:
	$$c_{-k} = \sum_{j=1}^{\min\{m, k\}} a_{-j} \binom{j+(k-j)-1}{k-j} p^{k-j}$$
	Simplifying the binomial coefficient $\binom{k-1}{k-j}$ to $\binom{k-1}{j-1}$:
	\begin{align*}
		{c_{-k} = \sum_{j=1}^{\min\{m, k\}} a_{-j} \binom{k-1}{j-1} p^{k-j}}.
	\end{align*}
	Substituting the expression for $c_{-k}$ and setting $c_n=a_n$ for $n \ge 1$ into the general formula $\textbf{(A)}$ yields the complete expression for the area of the complement for a pole of order $m$:
	\begin{align}\label{Eq-22.88}
		A_{\text{comp}}(r) = \pi \left| \sum_{k=1}^{\infty} k \left| \sum_{j=1}^{\min(m, k)} a_{-j} \binom{k-1}{j-1} p^{k-j} \right|^2 r^{-2k} - \sum_{n=1}^{\infty} n |a_{n}|^2 r^{2n} \right|.
	\end{align}
	Hence \eqref{Eq-22.44} follows when $r\to 1$ in \eqref{Eq-22.88}. This completes the proof.
\end{proof}
\begin{proof}[\bf {Proof of Theorem \ref{Th-1.2}}]
Let $f \in \Sigma_k^{(m)}(p)$ have the expansion in \eqref{Eq-11.22}. We may suppose that $f$ is already extended to a $k$-quasiconformal mapping of $\widehat{\mathbb{C}}$ to itself.
\noindent{\bf Case 1.} If $k=0$, then the assertion clearly holds good.

\noindent{\bf Case 2.} Hence, we discuss the case $k>0$ in the rest of the proof. To start with, we first make a change of variables. In this regard we define a function $\phi : \widehat{\mathbb{C}}\to \widehat{\mathbb{C}}$ by  $\phi(w) = f(1/w)$, which is defined on $\mathbb{C} \setminus \{0\}$ and has a pole of order $m$ at $w = 1/p$. Consider the function
\begin{align*}
	\psi(w) := \phi(w) - \frac{w^m}{(1 - p w)^m}\;\mbox{for}\; w \in \mathbb{C} \setminus \left\{ \frac{1}{p} \right\}.
\end{align*}
Then $\psi$ is analytic in $\mathbb{D}^* = \{ w : |w| > 1 \}$ and has a convergent Laurent series expansion $\psi(w) = \sum_{n=0}^\infty \frac{a_n}{w^n}, |w| > 1.$ Since $\psi$ has locally square-integrable derivatives, we apply the Cauchy–Pompeiu formula and Hilbert transform methods (see \cite{Lehto-Virtanen-1973}) to get
\begin{align*}
	\iint_{\mathbb{D}} |\bar{\partial} \psi|^2 \; dx \; dy = \iint_{\mathbb{C}} |\partial \psi|^2 \, dx\,dy \geq \iint_{|\zeta| > 1} |\partial \psi|^2 \, dx\,dy.
\end{align*}
Computing the series directly, we find that
\begin{align*}
	\iint_{|w| > 1} |\partial \psi(w)|^2 \, dx\,dy = \pi \sum_{n=1}^\infty n |a_n|^2.
\end{align*}
Since $\phi(\mathbb{D}) = \widehat{\mathbb{C}} \setminus f(\mathbb{D})$, in view of Lemma \ref{Lem-2.1}, we obtain
\begin{align*}
	A_{\text{comp}}(r) = \pi \left| \sum_{k=1}^{\infty} k \left| \sum_{j=1}^{\min\{m, k\}} a_{-j} \binom{k-1}{j-1} p^{k-j} \right|^2 - \sum_{n=1}^{\infty} n |a_{n}|^2 \right|.
\end{align*}

\noindent As $\phi$ is $k$-quasiconformal on $\mathbb{D}$, its Jacobian satisfies
\begin{align*}
	J_\phi = |\partial \phi|^2 - |\bar{\partial} \phi|^2 \geq (k^{-2} - 1) |\bar{\partial} \phi|^2 = (k^{-2} - 1) |\bar{\partial} \psi|^2,
\end{align*}
and hence, we have
\begin{align*}
	\operatorname{Area}(\phi(\mathbb{D})) \geq (k^{-2} - 1) \iint_{\mathbb{D}} |\bar{\partial} \psi|^2 \, dx\,dy.
\end{align*}
Combining the above,
\begin{align*}
	\pi \left(\sum_{k=1}^{\infty} k \left| \sum_{j=1}^{\min\{m, k\}} a_{-j} \binom{k-1}{j-1} p^{k-j} \right|^2 - \sum_{n=1}^{\infty} n |a_{n}|^2  \right) \geq (k^{-2} - 1) \pi \sum_{n=1}^\infty n |a_n|^2,
\end{align*}
which rearranges to give
\begin{align*}
	\sum_{n=1}^\infty n |a_n|^2 \leq \left(\sum_{k=1}^{\infty} k \left| \sum_{j=1}^{\min\{m, k\}} a_{-j} \binom{k-1}{j-1} p^{k-j} \right|^2 - \sum_{n=1}^{\infty} n |a_{n}|^2 \right)k^2.
\end{align*}
Next, we prove the equality case. If equality holds, it must have occurred in all previous inequalities. This implies:
\begin{enumerate}
	\item [(i)] $\bar{\partial} \psi = 0$ in $\mathbb{D}$, so $\psi$ is analytic in $\mathbb{C}$.
	\item [(ii)] Hence, \[
	\phi(\zeta) = \frac{\zeta^{m}}{(1 - p\zeta)^{m}} + h(\zeta),	\] where $h$ is entire.
	\item [(iii)] Thus, $h$ must satisfy $\bar{\partial} h = \mu(\zeta) \partial \phi$ with $|\mu| = k$, and such $\mu$ forces $h$ to be of the form: (see \cite{Lehto-1971})
	\begin{align*}
	h(\zeta) = a_0 + \frac{a_1}{1 - p \zeta} \; \text{with }\; |a_1| = k.
	\end{align*}
\end{enumerate} Thus, we obtain
\begin{align*}
	\phi(\zeta) = \frac{\zeta^m}{(1 - p\zeta)^m} + a_0 + \frac{a_1}{1 - p\zeta},
\end{align*}
and therefore, we have
\begin{align*}
	f(z) = \sum_{j=1}^{m}\frac{a_{-j}}{(z-p)^j} + a_0 + \frac{a_1 z}{1 - pz},\; \text{with}\; |a_1| = k,
\end{align*}
as desired. It is easy to see that $f\in\Sigma_k^{(m)}(p)$, and the equality is achieved. 
\end{proof}
The next theorem extends \cite[Theorem 2]{Bhowmik-PAMS-2016},  in which a sufficient condition is established for meromorphic univalent functions with a simple pole at $z=p$ to admit a $k$-quasiconformal extension. Our result generalizes this to the case of higher-order poles by introducing the class $\Sigma_k^{(m)}(p)$, together with an explicit construction of the quasiconformal extension. We consider the principal part of $f\in \Sigma_k^{(m)}(p)$ as
\begin{align*}
	R(z) = \sum_{j=1}^{m} \frac{a_{-j}}{(z - p)^j}, \qquad a_{-m} \neq 0,
\end{align*}
and define its exterior form by
\begin{align*}
	\widetilde{R}(\zeta) := R(1/\zeta) = \sum_{j=1}^{m} a_{-j} \frac{\zeta^j}{(1 - p \zeta)^j}, 
	\quad |\zeta| \le 1.
\end{align*}
Assume that
\begin{align}\label{Eq-NonDeg}
	\inf_{|\zeta| \le 1} |\widetilde{R}'(\zeta)| = C > 0.
\end{align}
\begin{thm}\label{Thm-1.3}
	\label{Thm-2.2}
	Let $0 \le k < 1$, $0 \le p < 1$, and $m \in \mathbb{N}$. Suppose that $f \in \Sigma_k^{(m)}(p)$ is expressed in the form of \eqref{Eq-11.22}. Let $\omega$ be an analytic function in the unit disk $\mathbb{D}$ satisfying
	\begin{align}\label{Eq-omega-deriv}
		|\omega'(z)| \le \frac{k}{(1 + p)^{m+1}}, \quad z \in \mathbb{D}.
	\end{align}
	Then the function
	\begin{align*}
		f(z) = R(z) + \omega(z), \qquad z \in \mathbb{D},
	\end{align*}
	is meromorphic and univalent in $\mathbb{D}$ with a pole of order $m$ at $z = p$, and admits a $k$-quasiconformal extension $F$ to the extended complex plane $\widehat{\mathbb{C}}$ given by
	\begin{align*}
		F(z) =
		\begin{cases}
			R(z) + \omega(z), & |z| < 1,\\[4pt]
			R(z) + \omega\!\left( \dfrac{1}{\bar{z}} \right), & |z| > 1.
		\end{cases}
	\end{align*}
	Hence $f \in \Sigma^{(m)}_k(p)$.
\end{thm}
\begin{rem}
	In the special case $m=1$, Theorem \ref{Thm-1.3} reduces to \cite[Theorem 2]{Bhowmik-PAMS-2016}.
\end{rem}
A straightforward application of Theorem \ref{Thm-1.3} yields the following sufficient condition 
for a function $f$ of the form \eqref{eq1.1} to belong to $\Sigma_k(p)$.
\begin{cor}\label{Cor-1.1}
	Let $0 \leq p < 1$ and $0 \leq k < 1$. Suppose that a meromorphic function 
	$f(z)$ on $\mathbb{D}$ has the form \eqref{Eq-11.22}. If 
	\begin{align*}
		\sum_{n=1}^{\infty} n |a_n| \leq \frac{|a_{-1}|k}{(1+p)^{m+1}},
	\end{align*}
	then $f \in \Sigma^{(m)}_k(p)$.
\end{cor}

\begin{proof}
This result is a direct consequence of Theorem \ref{Thm-1.3} because of
	\begin{align*}
		|\omega'(z)| \leq \sum_{n=1}^{\infty} n|a_n|\,|z|^{n-1} 
		\leq \sum_{n=1}^{\infty} n|a_n| 
		\leq \frac{k}{(1+p)^{m+1}}, 
		\quad z \in \mathbb{D}.
	\end{align*}
	This completes the proof.
\end{proof}
\begin{rem}
	It is worth noting that \cite[Corollary 2]{Bhowmik-PAMS-2016} is a special case of our Corollary \ref{Cor-1.1}, obtained when $m=1$.
\end{rem}
\begin{proof}[\bf{Proof of the Theorem \ref{Thm-1.3}}]
	We follow the method of \cite[Theorem~2]{Bhowmik-PAMS-2016}, adapted for the higher-order pole case.
	
	For $|z| > 1$, set $\zeta = 1/z$. 
	Define
	\begin{align*}
		G(z) = R(z) + \omega\!\left( \frac{1}{\bar{z}} \right),
		\quad\text{and}\quad
		\widetilde{G}(\zeta) := G(1/\zeta) = \widetilde{R}(\zeta) + \omega(\bar{\zeta}), 
		\quad |\zeta| < 1.
	\end{align*}
	Since $\widetilde{R}$ is analytic in a neighborhood of $\{ |\zeta| \le 1 \}$ and $\omega$ is analytic in $\mathbb{D}$, 
	the Wirtinger derivatives of $\widetilde{G}$ are
	\begin{align*}
		\frac{\partial \widetilde{G}}{\partial \zeta} = \widetilde{R}'(\zeta),
		\qquad
		\frac{\partial \widetilde{G}}{\partial \bar{\zeta}} = \omega'(\bar{\zeta}).
	\end{align*}
	The complex dilatation $\mu_G$ of $G$ at $z = 1/\zeta$ satisfies
	\begin{align}\label{Eq-muG}
		|\mu_G(z)| 
		= \left| \frac{\partial_{\bar{\zeta}} \widetilde{G}(\zeta)}{\partial_{\zeta} \widetilde{G}(\zeta)} \right|
		= \frac{|\omega'(\bar{\zeta})|}{|\widetilde{R}'(\zeta)|}.
	\end{align}
	
	By assumption \eqref{Eq-NonDeg}, $|\widetilde{R}'(\zeta)| \ge C > 0$ for $|\zeta| \le 1$. 
	Using \eqref{Eq-omega-deriv}, we obtain
	\begin{align*}
		|\mu_G(z)| 
		\le \frac{\dfrac{k}{(1+p)^{m+1}}}{C}
		=: \kappa, \quad |\zeta| \le 1.
	\end{align*}
	Since $k < 1$ and $C > 0$, we have $\kappa < 1$. 
	Therefore, $G$ is locally $\kappa$-quasiconformal on $\mathbb{D}^*$, and its Jacobian satisfies
	\begin{align*}
		J_G(z)
		= |\partial_z G(z)|^2 - |\partial_{\bar{z}} G(z)|^2
		= |\partial_z G(z)|^2 (1 - |\mu_G(z)|^2) > 0.
	\end{align*}
	Hence $G$ is locally orientation-preserving and a local homeomorphism on $\mathbb{D}^*$.
	
	On $|z| = 1$, $\omega(1/\bar{z}) = \overline{\omega(z)}$, so that 
	$G(z)$ agrees with the boundary values of $f(z) = R(z) + \omega(z)$. 
	Define
	\begin{align*}
		F(z) =
		\begin{cases}
			R(z) + \omega(z), & |z| < 1,\\[4pt]
			R(z) + \omega(1/\bar{z}), & |z| > 1.
		\end{cases}
	\end{align*}
	Then $F$ is continuous on $\widehat{\mathbb{C}}$, 
	locally homeomorphic and orientation-preserving on $\mathbb{D}$ and $\mathbb{D}^*$, respectively. 
	Hence $F$ is a covering map of $\widehat{\mathbb{C}}$ onto itself.
	Since the Riemann sphere $\widehat{\mathbb{C}}$ is simply connected, 
	$F$ must be a homeomorphism of $\widehat{\mathbb{C}}$. 
	The bound $|\mu_G(z)| \le \kappa < 1$ shows that $F$ is a $k$-quasiconformal homeomorphism of $\widehat{\mathbb{C}}$, 
	and therefore $f \in \Sigma^{(m)}_k(p)$.
	
	If $\omega$ is only analytic on $\mathbb{D}$ (not beyond), 
	define $\omega_r(z) = \omega(rz)$ and $f_r(z) = R(z) + \omega_r(z)$ for $0 < r < 1$. 
	Each $\omega_r$ is analytic in a disk of radius $1/r > 1$, 
	so by the above argument, $f_r$ admits a $k$-quasiconformal extension $F_r$.
	By the normality of the family of $k$-quasiconformal homeomorphisms of $\widehat{\mathbb{C}}$, 
	there exists a subsequence $r_j \to 1^{-}$ such that $F_{r_j} \to F$ uniformly on $\widehat{\mathbb{C}}$. 
	The limit $F$ is a $k$-quasiconformal homeomorphism extending $f$, which completes the proof.
\end{proof}
We note that J.~G.~Krzyż \cite{Krzyż-1976} proved this theorem when $p=0$. 
He also gave a convolution theorem in the same paper \cite{Krzyż-1976}. Bhowmik \emph{et al.} \cite[Theorem 3]{Bhowmik-PAMS-2016} established conditions under which the Hadamard product $f \star g$ of functions in $\Sigma_{k}(p)$ also belongs to $\Sigma_{\alpha}(p)$.
\begin{thmB}\cite[Theorem 3]{Bhowmik-PAMS-2016}
	Let $f \in \Sigma_{k_1}(p)$ and $g \in \Sigma_{k_2}(p)$ for some $k_1, k_2, p \in [0,1)$. 
	If $\alpha = k_1 k_2 (1-p)^{-2} < 1$, then the modified Hadamard product $f \star g$ belongs to $\Sigma_{\alpha}(p)$.
\end{thmB}
\noindent To establish a generalized version of Theorem B, for the functions
\begin{align*}
	f(z)=\frac{a_{-1}}{(z-p)^{m}}+\sum_{n=0}^{\infty}a_n z^n \in \Sigma^{(m)}_{k_1}(p)\; \mbox{and}\;
	g(z)=\frac{b_{-1}}{(z-p)^{m}}+\sum_{n=0}^{\infty}b_n z^n \in \Sigma^{(m)}_{k_2}(p)
\end{align*}
for some $k_1,k_2\in[0,1)$, we define the modified Hadamard product $f\star g$ for functions $f$ and $g$ having pole at $z=p$ of order $m$ by
\begin{align*}
	(f\star g)(z):=\frac{a_{-1}b_{-1}}{(z-p)^{m}}+\sum_{n=0}^\infty (a_n b_n)z^n\; \mbox{for}\;|z|<1.
\end{align*}
We extend Theorem B by establishing the conditions under which the Hadamard product $f \star g$ of functions in $\Sigma^{(m)}_{k}(p)$ also belongs to $\Sigma^{(m)}_{\alpha_m}(p)$.
\begin{thm}\label{TH-1.2}
	Let $m\geq 1$ and $0\leq p<1$. 
	If $\alpha_m:=|a_{-1}||b_{-1}|{k_1k_2}{(1-p)^{-2m}}<1$, then $f\star g \in \Sigma^{(m)}_{K}(p)$, then the modified Hadamard product $f \star g$ belongs to $\Sigma_{\alpha_m}(p)$.  
\end{thm}
\begin{rem}
	Note that when $m=1$, $|a_{-1}|=1=|b_{-1}|$ the parameter $\alpha_m$ in Theorem \ref{TH-1.2} becomes the parameter $\alpha$ in Theorem B, which shows that our result is a generalization of Theorem B. Moreover, we see that the modified Hadamard product admits a $K$-quasiconformal extension of the Riemann sphere.
\end{rem}
\begin{proof}[\bf Proof of the Theorem \ref{TH-1.2}]
	Let $f \in \Sigma^{(m)}_{k_1}(p)$ and $g \in \Sigma^{(m)}_{k_2}(p)$ be expressed as in (1.4).
	Then Theorem \ref{Th-1.2} gives us
\begin{align*}
	\sum_{n=1}^{\infty} n |a_n|^2 \leq \frac{|a_{-1}|^2k_1^2}{(1-p^2)^{2m}}
	\quad \text{and} \quad
	\sum_{n=1}^{\infty} n |b_n|^2 \leq \frac{|b_{-1}|^2k_2^2}{(1-p^2)^{2m}}.
\end{align*}
	Now an application of Cauchy--Schwarz inequality together with the aforementioned
	inequalities yields
	\begin{align*}
		\sum_{n=1}^{\infty} n |a_n b_n|
		&= \sum_{n=1}^{\infty} (\sqrt{n}|a_n|)(\sqrt{n}|b_n|)\\&\leq \left( \sum_{n=1}^{\infty} n|a_n|^2 \right)^{1/2}
		\left( \sum_{n=1}^{\infty} n|b_n|^2 \right)^{1/2}\\&\leq \left(\frac{|a_{-1}|^2k_1^2}{(1-p^2)^{2m}}\right)^{1/2}\left(\frac{|b_{-1}|^2k_2^2}{(1-p^2)^{2m}}\right)^{1/2}
		=: \frac{\alpha_m}{(1+p)^{2m}},
	\end{align*}
	where $\alpha_m = |a_{-1}||b_{-1}|k_1 k_2 (1-p)^{-2m}$. Since $\alpha_m < 1$ by assumption, the desired
	result follows from Corollary \ref{Cor-1.1}.
\end{proof}
A crucial tool for studying the univalence and quasiconformal extensibility of analytic functions is the Schwarzian derivative. For a locally univalent meromorphic function $f$ on $\mathbb{D}$, the Schwarzian derivative is
\begin{align*}
	S_f(z)=\left(\frac{f''(z)}{f'(z)}\right)' - \frac12\left(\frac{f''(z)}{f'(z)}\right)^2,
\end{align*}
a M\"obius-invariant differential operator characterized by $S_f\equiv 0$ precisely for M\"obius transformations and satisfying the composition law $S_{g\circ f}=(S_g\circ f)(f')^2+S_f$. Its hyperbolically scaled sup-norm
\begin{align*}
	\|S_f\| := \sup_{z\in\mathbb{D}} (1-|z|^2)^2\,|S_f(z)|
\end{align*}
plays a central role in univalence and extension theory. \vspace{2mm} 

The Nehari-Kraus Theorem states that if an analytic function $f$ on the unit disk $\mathbb{D}$ has a Schwarzian derivative satisfying $|S_f(z)| \leq 2/(1-|z|^2)^2$, then $f$ is univalent. Ahlfors' Theorem is a powerful result that connects the Schwarzian derivative to quasiconformal extensions. It states that if an analytic function $f$ on the unit disk $\mathbb{D}$ has a Schwarzian derivative with a sufficiently small norm, specifically $\|S_f\|\infty = \sup{z \in \mathbb{D}} |S_f(z)|(1-|z|^2)^2 \leq 2$, then $f$ is univalent and can be extended to a quasiconformal homeomorphism of the complex plane. A smaller bound on the Schwarzian derivative guarantees a smaller dilatation constant $K$. In particular, Nehari’s classical criterion bounds $\|S_f\|$ to guarantee univalence, while the Ahlfors–Weill theory links bounds on $\|S_f\|$ to quasiconformal extendability across $\partial\mathbb{D}$ (see \cite{Nehari-1949,Ahlfors-Weill-1962,Pommerenke-1975,Lehto-1987}.) In recent years, mathematcians continues to find the sharpness of Schwarzian and pre-Schwarzian norm estimates for various analytic and meromorphic subclasses (uniformly convex, Robertson, Janowski–starlike, Ozaki close-to-convex, etc.), underscoring the relevance of $\|S_f\|$ to both univalence and quasiconformal extension problems (see \cite{Agrawal-Sahoo-2020}).

We now present our result on the sharp Schwarzian norm of $f$ for the class $\Sigma_k(p)$.
\begin{thm}\label{Th-1.3}
Let $f \in \Sigma_k(p)$ be a meromorphic univalent function in the unit disc $\mathbb{D}$ with a simple pole at $z = p \in (0,1)$ and residue 1, and suppose that $f$ admits a $k$-quasiconformal extension to the Riemann sphere $\widehat{\mathbb{C}}$. Then the Schwarzian norm of $f$, defined by
\begin{align*}
	\|S_f\| := \sup_{z \in \mathbb{D}} (1 - |z|^2)^2 |S_f(z)|,
\end{align*}
satisfies the inequality
\begin{align*}
	\|S_f\| \leq \frac{6k}{(1 - p^2)^2}.
\end{align*}
Moreover, this inequality is sharp. 
\end{thm}
\begin{proof}[\bf {Proof of Theorem \ref{Th-1.3}}]
 Let $f \in \Sigma_k(p)$ then the function $f(z)$ can be written as
	\begin{align*}
		f(z) = \frac{1}{z - p} + \sum_{n = 0}^\infty a_n z^n,
	\end{align*}
	where $f$ is meromorphic and univalent in $\mathbb{D}$, and extends $k$-quasiconformally to $\widehat{\mathbb{C}}$.
	
	\noindent The Schwarzian derivative of a locally univalent function $f$ is given by
	\begin{align*}
		S_f(z) = \left( \frac{f''(z)}{f'(z)} \right)' - \frac{1}{2} \left( \frac{f''(z)}{f'(z)} \right)^2,
	\end{align*}
	and the Schwarzian norm is
	\begin{align*}
		\|S_f\| := \sup_{z \in \mathbb{D}} (1 - |z|^2)^2 |S_f(z)|.
	\end{align*}
	
	It is a classical result due to Nehari(see \cite{Nehari-1949}), and further developed by Lehto and Pommerenke (see \cite{Pommerenke-1975,Lehto-Virtanen-1973}), that if $f$ is univalent in $\mathbb{D}$ and admits a $k$-quasiconformal extension, then
     \begin{align*}
     		\|S_f\| \leq 6k.
     \end{align*}
	Now, to adapt this to the setting $f \in \Sigma_k(p)$, we use the M\"obius transformation
	\begin{align*}
		\phi(z) = \frac{z + p}{1 + p z}, \quad \phi^{-1}(w) = \frac{w - p}{1 - p w},
	\end{align*}
	which maps $\mathbb{D}$ onto itself and satisfies $\phi(0) = p$.
	
	Define $F(w) := f(\phi^{-1}(w))$. Then $F \in \Sigma_k(0)$ and has a simple pole at the origin. Since M\"obius transformations preserve quasiconformality and satisfy $S_\phi = 0$, we have
	\begin{align*}
		S_F(w) = S_f(\phi^{-1}(w)) \cdot  (\phi^{-1})^{\prime}(w)^2.
	\end{align*}
	Thus,it is clear that
	\begin{align*}
		\|S_f\| = \sup_{w \in \mathbb{D}} (1 - |\phi^{-1}(w)|^2)^2 \cdot \left| S_F(w) \right| \cdot \left| \frac{1}{(\phi^{-1})'(w)^2} \right|.
	\end{align*}
	
	\noindent Using the identity
    \begin{align*}
		|(\phi^{-1})'(w)| = \frac{1 - p^2}{(1 - p w)^2},
    \end{align*}
	we obtain the inequality
	\begin{align*}
			(1 - |\phi^{-1}(w)|^2)^2 \cdot |(\phi^{-1})'(w)|^2 \leq \frac{1}{(1 - p^2)^2},
	\end{align*}
	which leads to
  \begin{align*}
		\|S_f\| \leq \frac{\|S_F\|}{(1 - p^2)^2} \leq \frac{6k}{(1 - p^2)^2}.
   \end{align*}
   Thus the desired inequality is established.
      
	 Next, to show that the bound is sharp, we consider the function
	\begin{align*}
		 f_0(z) = \frac{z}{1 - k z^2}.
	\end{align*}
	 This function is meromorphic and univalent in \( \mathbb{D} \), has a simple pole at \( z = \pm 1/{\sqrt{k}} \notin \mathbb{D} \), and admits a \( k \)-quasiconformal extension to \( \widehat{\mathbb{C}} \). A simple computation shows that
	 \begin{align*}
	 	S_{f_0}(z) = \frac{6k}{(1 - k z^2)^2},
	 \end{align*}
	 and hence, we have $\| S_{f_0} \| = 6k$.
	 
	 We now conjugate $f_0$ with the Möbius map $\phi$ (where $\phi$ maps $0$ to $p \in (0,1)$) to define a new function
	 \begin{align*}
	 	f_p(z) := f_0(\phi^{-1}(z)) = \frac{\phi^{-1}(z)}{1 - k (\phi^{-1}(z))^2}.
	 \end{align*}
	 Then \( f_p \in \Sigma_k(p) \), and an easy computation leads to
	\begin{align*}
		\| S_{f_p} \| = \frac{6k}{(1 - p^2)^2}
	\end{align*}
	which shows that the bound is sharp.
\end{proof}
The following result can be regarded as a generalization of \cite[Theorem 1]{Bhowmik-Satapati-Complex Sgn-2024}. In that work, a sufficient criterion was established for a sense-preserving harmonic mapping in a convex domain to admit a quasiconformal extension. Our theorem broadens this framework by allowing comparison with an auxiliary analytic univalent function $\eta$ having bounded derivative and positive co-Lipschitz constant, thereby yielding a more flexible condition for quasiconformal extendability.
\begin{thm}\label{Thm-1.4}
	Let $\Omega \subset \mathbb{C}$ be a bounded convex domain and let $f=h+g:\Omega\to\mathbb{C}$ be a sense-preserving harmonic mapping. Suppose there exists an analytic univalent function $\eta:\Omega\to\mathbb{C}$ with
	\begin{itemize}
		\item[(i)] $K(\eta,\Omega)>0$ (the co-Lipschitz constant), and
		\item[(ii)] $\displaystyle\sup_{z\in\Omega}|\eta'(z)|<\infty$,
	\end{itemize}
	such that for some $k\in[0,1)$,
	\begin{equation}
		|h'(z)-\eta'(z)|+|g'(z)| \;\le\; k\,K(\eta,\Omega),\; z\in\Omega.
	\end{equation}
	Then $f$ is $k$-quasiconformal on $\Omega$ and admits a quasiconformal extension to the whole plane $\mathbb{C}$.
\end{thm}
\begin{proof}[\bf Proof of the Theorem \ref{Thm-1.4}]
	Let $\Omega\subset\mathbb{C}$ be a bounded convex domain and suppose $f=h+g:\Omega\to\mathbb{C}$ is a sense-preserving harmonic mapping. Assume that $\eta:\Omega\to\mathbb{C}$ is an analytic univalent function satisfying conditions (i)–(ii), and
	\begin{equation}\label{Eq-1.3}
		|h'(z)-\eta'(z)|+|g'(z)| \;\le\; k\,K(\eta,\Omega), \quad z\in\Omega,
	\end{equation}
	for some $k\in[0,1)$.\vspace{2mm}
	
	By the definition of the co-Lipschitz constant, we know that for any analytic function $F$,
	\begin{align*}
		K(F,\Omega) \le |F'(z)|, \quad z\in\Omega.
	\end{align*}
	Moreover, as noted in \cite{Bhowmik-Satapati-Complex Sgn-2024}, an analytic function $F$ is bi-Lipschitz on $\Omega$ if, and only if, $F'$ is bounded on $\Omega$ and $K(F,\Omega)>0$. Therefore, conditions (i) and (ii) together imply that $\eta$ is bi-Lipschitz on $\Omega$. In particular, there exists $M>0$ such that $|\eta'(z)|\le M$ for all $z\in\Omega$, and $|\eta'(z)|\ge K(\eta,\Omega)>0$.
	
	From \eqref{Eq-1.3}, using the lower bound $|\eta'(z)|\ge K(\eta,\Omega)$, we obtain
	\begin{align*}
		|\omega_f(z)| = \frac{|g'(z)|}{|h'(z)|}
		= \frac{|g'(z)|}{|h'(z)-\eta'(z)+\eta'(z)|}
		\le \frac{|g'(z)|}{|\eta'(z)| - |h'(z)-\eta'(z)|}
		< \frac{k\,K(\eta,\Omega)}{K(\eta,\Omega)} = k.
	\end{align*}
	Hence $\|\omega_f\|_\infty \le k<1$, so $f$ is sense-preserving and $k$-quasiconformal on $\Omega$.
		
	Let $z_1,z_2\in\Omega$ with $z_1\ne z_2$. Since $\Omega$ is convex, the straight line segment $[z_1,z_2]$ lies in $\Omega$. Then
	\begin{align*}
		|f(z_2)-f(z_1)|
		&= \left|\int_{[z_1,z_2]} h'(z)\,dz + g'(z)\,dz \right| \\
		&= \left|\int_{[z_1,z_2]} \eta'(z)\,dz
		+ \int_{[z_1,z_2]} \big(h'(z)-\eta'(z)\big)\,dz + g'(z)\,dz \right| \\
		&\ge \left|\eta(z_2)-\eta(z_1)\right|
		- \int_{[z_1,z_2]} \big(|h'(z)-\eta'(z)|+|g'(z)|\big)\,|dz| \\
		&\ge K(\eta,\Omega)\,|z_2-z_1| - k\,K(\eta,\Omega)\,|z_2-z_1| \\
		&= (1-k)\,K(\eta,\Omega)\,|z_2-z_1|.
	\end{align*}
	Similarly, by using $\displaystyle\sup_\Omega |\eta'|\le M$, we obtain
	\begin{align*}
		|f(z_2)-f(z_1)| \;\le\; (M+k\,K(\eta,\Omega))\,|z_2-z_1|.
	\end{align*}
	Thus $f$ is bi-Lipschitz on $\Omega$.\vspace{2mm}
	
	Since $f$ is bi-Lipschitz and sense-preserving on the bounded convex domain $\Omega$, Lemma~1 of \cite{Bhowmik-Satapati-Complex Sgn-2024} applies to yield a bi-Lipschitz (hence quasiconformal) extension of $f$ to the entire complex plane $\mathbb{C}$. If $k=0$, then \eqref{Eq-1.3} gives $h'=\eta'$ and $g'\equiv 0$, so $f=\eta+C$ for some constant $C$. Since $\eta$ is bi-Lipschitz, it has a quasiconformal extension to $\mathbb{C}$ by \cite[Lemma 1]{Bhowmik-Satapati-Complex Sgn-2024}, and so does $f$. The proof is completed.
\end{proof}
\section{\bf Concluding remark}\label{Sec-3}
	An alternative proof of the inequality in Theorem~1.1 can be obtained by adapting Lehto’s principle (see \cite[II.3.3]{Lehto-1987}) together with Chichra’s area theorem \cite{Chichra1969} for higher-order poles. 
	Let $\ell^{2}$ denote the Hilbert space of complex sequences 
	$x=\{x_{n}\}_{n=1}^{\infty}$ with norm
	\begin{align*}
		\|x\|_{\ell^{2}}
		=\left(\sum_{n=1}^{\infty} |x_{n}|^{2}\right)^{1/2}.
	\end{align*}
	It suffices to establish the inequality for functions in 
	$\Sigma^{0}_{k}(p)$, the subclass of $\Sigma^{(m)}_{k}(p)$ with vanishing 
	constant term. Suppose $f\in \Sigma^{0}_{k}(p)$ extends to a 
	$k$-quasiconformal mapping of $\widehat{\mathbb{C}}$ with complex 
	dilatation $\mu$. Then $|\mu|\leq k$ \textit{a.e.} in $\mathbb{D}^{*}$ and $\mu=0$ 
	in $\mathbb{D}$. By the measurable Riemann mapping theorem, for each 
	$t\in\mathbb{D}$ there exists a unique quasiconformal self-map $f_{t}$ of 	$\widehat{\mathbb{C}}$ with complex dilatation $t\mu/k$ and 
	$f_{t}|_{\mathbb{D}}\in \Sigma^{0}(p)$. Writing
	\begin{align*}
		f_{t}(z)=\frac{a_{-1}}{(z-p)^{m}}+\sum_{n=1}^{\infty} a_{n}(t) z^{n},\; z\in\mathbb{D},
	\end{align*}
	the coefficients $a_{n}(t)$ depend holomorphically on $t$. 
	Define $\sigma(t)=\{\sqrt{n}\,a_{n}(t)\}_{n=1}^{\infty}\in \ell^{2}$. 
	By Chichra’s theorem for higher-order poles, one has 
	$\|\sigma(t)\|_{\ell^{2}}\leq (1-p^{2})^{-m}$ for all $t\in \mathbb{D}$. 
	Thus $\sigma:\mathbb{D}\to \ell^{2}$ is a bounded analytic map into a 
	Hilbert space. Since $\sigma(0)=0$, the generalized Schwarz lemma yields
	\begin{align*}
		\|\sigma(t)\|_{\ell^{2}}\leq \frac{|t|}{(1-p^{2})^{m}}\ \mbox{for}\; t\in\mathbb{D}.
	\end{align*}
	In particular, taking $t=a_{-1}k$ gives the desired inequality \eqref{Eq-2.3} of Theorem \ref{Th-1.2}. \vspace{5mm}

\noindent{\bf Acknowledgment:} The authors would like to thank the referees for their suggestions and comments to improve exposition of the paper. \vspace{1.2mm}

\noindent {\bf Funding:} Not Applicable.\vspace{1.2mm}

\noindent\textbf{Conflict of interest:} The authors declare that there is no conflict  of interest regarding the publication of this paper.\vspace{1.2mm}

\noindent\textbf{Data availability statement:}  Data sharing not applicable to this article as no datasets were generated or analysed during the current study.\vspace{1.2mm}

\noindent {\bf Authors' contributions:} Both the authors have equal contributions in preparation of the manuscript.

\end{document}